\documentclass{amsart}
\usepackage{amsthm,amsmath,amsfonts,amssymb}
\usepackage{color}
\usepackage[english]{babel}
\usepackage[colorlinks=true,linkcolor=blue,citecolor=blue]{hyperref}
\usepackage[all]{xy}
\usepackage{verbatim}
\DeclareMathOperator*{\colim}{colim}

\providecommand{\eprint}[2][]{\href{http://arxiv.org/abs/#2}{arXiv:#2}}

\newtheorem{theorem}{Theorem}[section]
\newtheorem{lemma}[theorem]{Lemma}
\newtheorem{proposition}[theorem]{Proposition}
\newtheorem{corollary}[theorem]{Corollary}

\theoremstyle{definition}
\newtheorem{definition}[theorem]{Definition}
\newtheorem{example}[theorem]{Example}

\theoremstyle{remark}
\newtheorem{remark}[theorem]{Remark}

\newcommand{\xra}{\xrightarrow}
\newcommand{\xla}{\xleftarrow}

\newcommand{\Hom}{\operatorname{Hom}}
\newcommand{\coCh}{\operatorname{\textbf{coCh}}}

\newcommand{\AR}{\operatorname{\textbf{AR}}}
\newcommand{\Contr}{\operatorname{\textbf{Contr}}}

\newcommand{\Id}{\operatorname{Id}}

\newcommand{\Map}{\operatorname{Map}}

\newcommand{\Z}{\mathbb{Z}}
\newcommand{\N}{\mathbb{N}}

\title{The projective model structure on contractions}

\author{Marco Manetti}
\address{\newline
Universit\`a degli studi di Roma La Sapienza,\hfill\newline
Dipartimento di Matematica \lq\lq Guido
Castelnuovo\rq\rq,\hfill\newline
P.le Aldo Moro 5,
I-00185 Roma, Italy.}
\email{manetti@mat.uniroma1.it}
\urladdr{www.mat.uniroma1.it/people/manetti/}

\author{Chiara Spagnoli}
\email{chiaraspagnoli93@hotmail.it}

\subjclass[2010]{55U15, 55U35}
\keywords{Contractions of cochain complexes, Model categories}

\begin{document}

\begin{abstract} We prove that the projective model structure on the category of unbounded cochain complexes 
extends naturally to the category of contractions. The proof is completely elementary and we do not assume 
familiarity with model categories. 
\end{abstract}

\maketitle


\section*{Introduction}

Let $R$ be a commutative ring, a contraction of cochain complexes of $R$-modules is a diagram
\begin{equation}\label{equ.contraction}
\xymatrix{M\ar@<.4ex>[r]^\imath&N\ar@<.4ex>[l]^\pi\ar@(ul,ur)[]^h}
\end{equation}
where $M,N$ are (unbounded) cochain complexes of $R$-modules,  
$\imath,\pi$ are morphisms of cochain complexes  and $h$ is an $R$-linear map of degree $-1$  such that:
\begin{enumerate}

    \item (deformation retraction) $\;\pi\imath=\Id_{M}$,
    $\;\imath\pi-\operatorname{Id}_{N}=d_{N}h+hd_{N}$;

    \item (side conditions) $\;\pi h=h\imath=hh=0$.
\end{enumerate}
The notion of contraction was introduced by  Eilenberg and Mac Lane \cite{eilmactw} and plays a central role in homological perturbation theory \cite{gug72,HK,LS} and homotopy transfer of infinity structures 
\cite{FMcone,getzler04,HK,perturbation}. A morphism of contractions is defined in the obvious way as a morphism of diagrams, and the category of contractions is denoted by $\mathbf{Contr}(R)$. 
The category $\mathbf{coCh}(R)$ of cochain complex can be interpreted as the full subcategory of $\mathbf{Contr}(R)$
consisting of contractions with $\imath=\pi=\Id$ and $h=0$, and there exists a  
faithful  functor
\[ \Phi\colon \mathbf{Contr}(R)\to \mathbf{coCh}(R),\qquad 
\xymatrix{M\ar@<.4ex>[r]^\imath&N\ar@<.4ex>[l]^\pi\ar@(ul,ur)[]^h}\;\mapsto\; N\,.\]

Since the explicit formulas of homotopy transfer, see e.g. \cite{HK}, commute with morphisms of contractions, 
it is natural to ask for a homotopy theory of contractions, and more specifically whether a given model structure on the category of cochain complexes extends to the category of contractions. In this paper we study this problem for the projective model structure, and we prove that there exists a model structure on $\mathbf{Contr}(R)$, where 
a morphism $f$ is a weak equivalence, fibration, cofibration if and only if $\Phi(f)$ is.  
For simplicity of exposition we restrict to unbounded complexes, although the same ideas can be applied, with minor modification, also to complexes in nonpositive degrees, as well as other model structures on $\mathbf{coCh}(R)$.    
As a byproduct of our proof we also prove that also the category of acyclic retractions carries  a natural model structure, where an acyclic retraction is defined as a diagram 
$\xymatrix{M\ar@<.4ex>[r]^\imath&N\ar@<.4ex>[l]^\pi}$, with both $\imath,\pi$ quasi-isomorphisms of complexes and $\pi\imath=\Id_M$.

The proof that we give is completely elementary and relies essentially on three algebraic results, called  ``basic tricks'' on contractions and acyclic retraction: the first two tricks are well known and widely present in literature, see e.g. \cite{LS}, while the third appears new, at least to our knowledge.

\bigskip
\section{A short review of model categories}
\label{sec.modelcategories}

For the benefit of the reader and for fixing notation, in this section we briefly recall the notion of model category and the definition of the projective model structure in the category of cochain complexes over a commutative unitary ring. The main reference is Hovey's book \cite{Hov99}.

For every category $\mathbf{C}$ we shall write $A\in \mathbf{C}$ if $A$ is an object of $\mathbf{C}$ and we denote by $\Hom_{\mathbf{C}}(A,B)$ the set of morphisms $A\to B$. We denote by  
$\Map(\mathbf{C})$ the category whose objects are morphism in $\mathbf{C}$ and whose morphisms are the  commutative squares. 
The following definition gives the basic terminology involved in the notion of model category.

\begin{definition} In the above notation:
\begin{enumerate}

\item A morphism $f$  is a called a \emph{retract} of a morphism $g$ if  there exists a commutative diagram of the form 
\[ \begin{matrix}\xymatrix{
A \ar[r] \ar@/^1pc/[rr]^{\Id_A} \ar[d]_{f} & C \ar[r] \ar[d]^g & A \ar[d]^f\\
B \ar[r] \ar@/_1pc/[rr]_{\Id_B} & D \ar[r] &B 
}\end{matrix}\;.\]

\item A \emph{functorial factorization} is an ordered pair $(H,K)$ of functors 
$\Map(\mathbf{C}) \to \Map(\mathbf{C})$ such that $f=K(f)\,H(f)$ for every $f \in 
\Map(\mathbf{C})$.

\item Let $i \in \Hom_{\mathbf{C}}(A,B)$ and $p \in \Hom_{\mathbf{C}}(X,Y)$. We shall say that $i$ has the \emph{left lifting property (LLP) with respect to p} and $p$ has the 
\emph{right lifting property (RLP) with respect to i} if for every commutative diagram of solid arrow
\[ \xymatrix{ A \ar[r]^f \ar[d]_i & X \ar[d]^p \\ B \ar[r]_g \ar@{.>}[ru]^h & Y}\]
there exists a morphism $h \colon B \to X$ such that $hi=f$ and $ph=g$.

\end{enumerate}
\end{definition}

\begin{definition}\label{def.ModelStructure} 
A \emph{model structure} on a category $\mathbf{C}$ is the data of three classes of morphisms 
called weak equivalences, cofibrations, and fibrations, and two functorial factorizations 
$(C,FW)$ and $(CW,F)$ satisfying the following properties:
\begin{description}
 \item[MC1] (2-out-of-3) If $f$ and $g$ are morphisms of $\mathbf{C}$ such that $gf$ is defined and two of $f$, $g$ and $gf$ are weak equivalences, then so is the third.
 \item[MC2] (Retracts) If $f$ and $g$ are morphisms of $\mathbf{C}$ such that $f$ is a retract of $g$ and $g$ is a weak equivalence, cofibration, or fibration, then so is $f$. 
 \item[MC3] (Lifting) Define a map to be a trivial cofibration if it is both a cofibration and a weak equivalence. Similarly, define a map to be a trivial fibration if it is both a fibration and a weak equivalence. Then trivial cofibrations have the left lifting property with respect to fibrations, and cofibrations have the left lifting property with respect to trivial fibrations.
 \item[MC4] (Factorization) For any morphism $f$, $C(f)$ is a cofibration, $FW(f)$ is a trivial fibration, $CW(f)$ is a trivial cofibration, and $F(f)$ is a fibration.
\end{description}
A \emph{model category} is a complete and cocomplete category $\mathbf{C}$ equipped with a model structure.
\end{definition}

It is easy to see that in every model category we have, see e.g. \cite{Hov99}:
\begin{itemize}

\item every isomorphism is both a trivial fibration and a trivial cofibration;

\item the classes of weak equivalences, cofibrations and fibrations are closed by composition;

\item the pull-back of a fibration (resp.: trivial fibration) under any morphism is a fibration (resp.: trivial fibration);

\item the push-out of a cofibration (resp.: trivial cofibration) under any morphism is a cofibration (resp.: trivial cofibration).
\end{itemize}

\bigskip

For every commutative unitary ring $R$ we shall denote by $\coCh(R)$ the category of  cochain complexes of $R$-modules. Every object is the data of a collection of $R$-modules $X=\{X^n\}_{n\in \mathbb{Z}}$ and a differential $d=\{d_n \colon X^n \to X^{n+1}\}_{n\in \mathbb{Z}}$, where each $d_n$ is an $R$-module map and $d_{n+1} d_n =0 $ for all $n \in \mathbb{Z}$.
A morphism of cochain complexes $f \colon X \to Y$ is a collection of 
morphisms of $R$-module  $f_n \colon X^n \to Y^n$ such that $d_n f_n = f_{n+1} d_n$.
A quasi-isomorphism of cochain complexes is a morphism that induces isomorphisms on all cohomology groups.

The category $\coCh(R)$ has all small limits and colimits, which are taken degreewise. The initial and terminal object is the trivial complex, which is $0$ in each degree. 
This category carries several different model structures, \cite{sixmodel,Hov99}: in this paper we only deal with the so called projective model structure on unbounded complexes, although our results can be easily extended also to other model structures and to bounded complexes. 

\begin{theorem} \label{thm.ModelcoCh} There is a model category structure on the category of chain complexes $\coCh(R)$ whose
\begin{itemize}
 \item  weak equivalences are quasi-isomorphisms;
 \item  fibrations are the morphisms that are degreewise epimorphisms;
 \item  cofibrations are the morphisms having the left lifting property with respect every trivial fibration;
\end{itemize}
called the \emph{projective model structure.}
\end{theorem}

\begin{proof} See e.g. either \cite[Thm. 2.3.11]{Hov99} or \cite[Thm. 1.4]{sixmodel}.\end{proof}

\begin{remark} There exists a more concrete description of cofibrations as retracts of  
semifree extensions, see Appendix~\ref{sec.semifree}. If $X\to Y$ is a cofibration then every map 
$X^i\to Y^i$ is injective with projective cokernel; the converse is true if there exists an integer $n$ such that $X^i\to Y^i$ is an isomorphism for every $i\ge n$.
\end{remark}


\bigskip
\section{Contractions and acyclic retractions}

Every cochain complex is intended over a fixed unitary commutative ring $R$.
If $N,M$ are cochain complexes and $n$ is an integer we shall denote by $\Hom_R^n(N,M)$ the $R$-module of sequences $\{f_i\}_{i\in \Z}$, where every $f_i\colon N^i\to M^{n+i}$ is a morphism of $R$-modules.

\begin{definition}\label{def.ARdata}
An  \emph{acyclic retraction} (AR-data) is a diagram 
\[ \xymatrix{M\ar@<.4ex>[r]^\imath&N\ar@<.4ex>[l]^\pi}\]
where $M,N$ are cochain complexes of $R$-modules and $\imath,\pi$ are quasi-isomorphisms of 
cochain complexes such that $\;\pi\imath=\Id_{M}$.
A \emph{morphism} of acyclic retractions
\[ f\colon
(\xymatrix{M\ar@<.4ex>[r]^\imath&N\ar@<.4ex>[l]^\pi})\to
(\xymatrix{A\ar@<.4ex>[r]^i&B\ar@<.4ex>[l]^p})\]
is a morphism of cochain complexes  $f\colon N\to B$ such that $f\imath\pi= i pf$.
\end{definition}

Given a morphism of acyclic retractions  as in Definition~\ref{def.ARdata}, we have a commutative diagram
\[\xymatrix{M\ar[r]^{\imath}\ar[d]^{\hat{f}}&N\ar[r]^{\pi}\ar[d]^f&M\ar[d]^{\hat{f}}\\
A\ar[r]^i&B\ar[r]^p&A}\]
where  $\hat{f}=pf\imath$:
in fact 
$i\hat{f}=ipf\imath=f\imath\pi\imath=f\imath$, 
$\hat{f}\pi=pf\imath p=pipf=pf$.

\begin{definition}\label{def.contraction}
A \emph{contraction} of cochain complexes  is a pair 
$(\xymatrix{M\ar@<.4ex>[r]^\imath&N\ar@<.4ex>[l]^\pi},h)$, where $\xymatrix{M\ar@<.4ex>[r]^\imath&N\ar@<.4ex>[l]^\pi}$ is an acyclic retraction and an element $h\in\Hom^{-1}_R(N,N)$, called homotopy, that satisfies the following conditions:
\begin{description}

\item[C1] $\imath\pi-\operatorname{Id}_{N}=d_{N}h+hd_{N}$;

\item[C2]  $\pi h=h\imath=0$;

\item[C3] $h^2=0$.
\end{description}
A morphism of contraction $f\colon
(\xymatrix{M\ar@<.4ex>[r]^\imath&N\ar@<.4ex>[l]^\pi}, h)\to
(\xymatrix{A\ar@<.4ex>[r]^i&B\ar@<.4ex>[l]^p}, k)$ is 
a morphism of cochain complexes $f\colon N\to B$ such that $fh=kf$.
\end{definition}

For later use, we point out that 
if $(\xymatrix{M\ar@<.4ex>[r]^\imath&N\ar@<.4ex>[l]^\pi}, h)$ is a contraction, then $hdh=-h$, since
\[  h+hdh=h+(\imath\pi-\Id_N-dh)h=h-h=0\,.\]

Every morphism of contractions is in particular a morphism of acyclic retractions: in the notation of Definition~\ref{def.contraction} we have 
\[ ipf=(\Id_B+kd+dk)f=f(\Id_N+kd+dk)=f\imath\pi\,.\]

Thus, denoting by $\mathbf{AR}(R)$ and $\mathbf{Contr}(R)$ the categories of acyclic retractions and 
contractions, we have two forgetful functors
\[  
\mathbf{Contr}(R)\xrightarrow{\;\alpha\;} \mathbf{AR}(R)\xrightarrow{\;\beta\;}\mathbf{coCh}(R),\]
\[ (\xymatrix{M\ar@<.4ex>[r]^\imath&N\ar@<.4ex>[l]^\pi},h)\;\mapsto\; 
\xymatrix{M\ar@<.4ex>[r]^\imath&N\ar@<.4ex>[l]^\pi}\;\mapsto\; N\,.\]

It is straightforward to check that the categories $\mathbf{AR}(R)$ and $\mathbf{Contr}(R)$ are complete and cocomplete.  We are now ready to state the main results of this paper.

\begin{theorem}\label{thm.modelstructureAR} 
There exists a model structure on the category $\mathbf{AR}(R)$ where a morphism $f$ is a weak equivalence, cofibration, fibration if and only if $\beta(f)$ is, and  
the factorizations depends functorially on the factorizations in $\mathbf{coCh}(R)$.
\end{theorem}

\begin{theorem}\label{thm.modelstructureContr} There exists a model structure on the category $\mathbf{Contr}(R)$ where a morphism $f$ is a weak equivalence, cofibration, fibration if and only if $\beta\alpha(f)$ is, and  
the factorizations depends functorially on the factorizations in $\mathbf{coCh}(R)$.
\end{theorem}

The proofs will be given in next sections after some preparatory algebraic results about  contractions.

\bigskip
\section{The basic tricks}
\label{sec.basictrick}

This section is devoted  some algebraic properties about contractions and acyclic retractions that will used in the proof of the main theorems.

\begin{lemma}[First basic trick]\label{lem.trick2} Let
$\xymatrix{M\ar@<.4ex>[r]^\imath&N\ar@<.4ex>[l]^\pi}$ and $\xymatrix{A\ar@<.4ex>[r]^i&B\ar@<.4ex>[l]^p}$ be acyclic retractions  and let $f\colon N\to B$ be a morphism of cochain complexes. Then 
\[ \hat{f}:=f-ipf-f\imath\pi+2ipf\imath\pi\colon N\to B\]
is a morphism of acyclic retractions. 
Moreover:
\begin{enumerate}
\item the morphism $f-\hat{f}=ipf(\Id-\imath\pi)+(\Id-ip)f\imath\pi$ induces the trivial morphism in cohomology,
\item $\hat{f}=f$ whenever $f$ is a morphism of acyclic retractions, 
\item if $g$ is a morphism of acyclic retractions and $gf$ (resp.: $fg$) is defined, then 
$\widehat{gf}=g\hat{f}$ (resp.: $\widehat{fg}=\hat{f}g$). 
\end{enumerate} 
\end{lemma}

\begin{proof} Easy and straightforward.\end{proof}

There exists in literature the the notion of strong deformation data, which lies in an intermediate position with respect to acyclic retractions and contractions.

\begin{definition}\label{def.morfismoSDR}
A \emph{strong deformation retraction} (SDR-data) of cochain complexes  is a pair 
$(\xymatrix{M\ar@<.4ex>[r]^\imath&N\ar@<.4ex>[l]^\pi},h)$, where $\xymatrix{M\ar@<.4ex>[r]^\imath&N\ar@<.4ex>[l]^\pi}$ is an acyclic retraction and $h\in\Hom^{-1}_R(N,N)$ satisfy 
$\imath\pi-\operatorname{Id}_{N}=d_{N}h+hd_{N}$.
A morphism of strong deformation retractions  $f\colon
(\xymatrix{M\ar@<.4ex>[r]^\imath&N\ar@<.4ex>[l]^\pi}, h)\to
(\xymatrix{A\ar@<.4ex>[r]^i&B\ar@<.4ex>[l]^p}, k)$ is 
a morphism of cochain complexes $f\colon N\to B$ such that $fh=kf$.
\end{definition}

Thus, if $\mathbf{SDR}(R)$ denotes the category of strong deformation retractions, we have that 
$\mathbf{Contr}(R)$ is a full subcategory of $\mathbf{SDR}(R)$. Every morphism of strong deformation retractions is also a morphism a acyclic retractions.

\begin{lemma}[Second basic trick]\label{lem.trick1} 
Let $(\xymatrix{M\ar@<.4ex>[r]^\imath&N\ar@<.4ex>[l]^\pi}, h)$ be a strong deformation retraction  and denote 
\[D(h):=hd+dh=\imath\pi-\Id_N\colon N\to N.\] 
Then  the pair
\[ (\xymatrix{M\ar@<.4ex>[r]^\imath&N\ar@<.4ex>[l]^\pi}, \tilde{h}),\quad\text{where}\quad
\tilde{h}=-D(h)hD(h)dD(h)hD(h),\]
is a contraction. If $(\xymatrix{M\ar@<.4ex>[r]^\imath&N\ar@<.4ex>[l]^\pi}, h)$ is already a contraction, then $h=\tilde{h}$.
\end{lemma}

\begin{proof} This is well known \cite{LS} and we write the proof only for completeness. We have the equalities  
\[dD(h)=D(h)d=dhd,\qquad D(h)^2=(\imath\pi-\Id_N)^2=\Id_N-\imath\pi=-D(h),\] 
\[D(h)\imath=(\imath\pi-\Id_N)\imath=0,\qquad 
\pi D(h)=\pi(\imath\pi-\Id_N)=0\,.\] 
Therefore, setting $k=D(h)hD(h)$, we get $k\imath=\pi k=0$ and
\[ dk+kd=dD(h)hD(h)+D(h)hD(h)d=D(h)(dh+hd)D(h)=D(h)^3=D(h).\]
By definition $\tilde{h}=-kdk$, therefore $\tilde{h}\imath=\pi\tilde{h}=0$,
\[ \tilde{h}=-kdk=k(\Id_N-\imath\pi+kd)=k+k^2d,\qquad \tilde{h}=(\Id_N-\imath\pi+dk)k=k+dk^2\,,\] 
and then $k^2d=dk^2$. Finally 
\[d\tilde{h}+\tilde{h}d=d(k+dk^2)+(k+k^2d)d=dk+kd=\imath\pi-\Id_N\,,\] 
\[ \tilde{h}^2=kdkkdk=kd(k^2d)k=kd(dk^2)k=0\,.\]
If $\pi h=h\imath=h^2=0$ then $D(h)h=hD(h)=-h$ and therefore $k=D(h)hD(h)=h$, $k^2=h^2=0$, 
$\tilde{h}=k+dk^2=k=h$. 
\end{proof}

It is plain that the second basic trick is functorial in the following sense: 
given a morphism of strong deformation retractions 
\[ f\colon
(\xymatrix{M\ar@<.4ex>[r]^\imath&N\ar@<.4ex>[l]^\pi}, h)\to
(\xymatrix{A\ar@<.4ex>[r]^i&B\ar@<.4ex>[l]^p}, k),\]
then $\tilde{k}f=f\tilde{h}$.

\begin{lemma}[Third basic trick]\label{lem.trick3} 
Let $(\xymatrix{M\ar@<.4ex>[r]^\imath&N\ar@<.4ex>[l]^\pi}, h)$, $(\xymatrix{A\ar@<.4ex>[r]^i&B\ar@<.4ex>[l]^p}, k)$ be contractions and let $f\colon N\to B$ be a morphism of acyclic retractions. Then 
$\tilde{f}=f-dkfhd$
is a morphism of contractions. Moreover:
\begin{enumerate}
\item the morphism $f-\tilde{f}$ is homotopic to $0$,
\item $\tilde{f}=f$ whenever $f$ is a morphism of contractions, 
\item the transformation  $f\mapsto \tilde{f}$ commutes with compositions.
\end{enumerate}  
\end{lemma}

\begin{proof} We first notice that, since $dkd=dip-d$ we have 
\[  dkfhd=dkf(\imath\pi-Id-dh)=dkipf-dkf-dkdfh=-dkf-dipfh+dfh=fdh-dkf,\]
and then $\tilde{f}=f-dkfhd=f+dkf-fdh$. In particular $\tilde{f}=f$ whenever $f$ is already a morphism of contractions.

It is clear that $d\tilde{f}=\tilde{f}d$, i.e., $\tilde{f}$ is a morphism of complexes, and that the morphism $f-\tilde{f}=dkfhd=d(kfhd)+(kfhd)d$ is homotopic to $0$. Since $hdh=-h$ and $kdk=-k$ we have 
\[ k\tilde{f}-\tilde{f}h=kf-fh-kdkfhd+dkfhdh=kf-fh+kfhd-dkfh\,.\]
Denoting $\gamma=kf-fh$, since $h\imath\pi=0$ and $kip=0$  we have $\gamma \imath\pi=kf\imath\pi=kip f=0$ and 
\[\begin{split} 
d\gamma+\gamma d&=dkf+kfd-dfh-fhd=(dk+kd)f-f(dh+hd)\\
&=(ip-I)f-f(\imath\pi-I)=0\,.\end{split}\]
This implies that $\tilde{f}$ is a morphism of contractions since 
\[\begin{split} k\tilde{f}-\tilde{f}h&=\gamma+kfhd-dkfh=\gamma+(\gamma+fh)hd-d(\gamma+fh)h=\gamma+\gamma hd-d\gamma h\\
&=\gamma (I+hd+dh)=\gamma \imath\pi=0\,.\end{split}\]
If   $(\xymatrix{M\ar@<.4ex>[r]^\imath&N\ar@<.4ex>[l]^\pi}, h)$, $(\xymatrix{A\ar@<.4ex>[r]^i&B\ar@<.4ex>[l]^p}, k)$ and $(\xymatrix{P\ar@<.4ex>[r]^j&Q\ar@<.4ex>[l]^q}, l)$ are contractions and $f\colon N\to B$, $g\colon B\to Q$ are morphisms of acyclic retractions we have 
\[ \begin{split}\tilde{g}\tilde{f}&=(g+dlg-gdk)(f+dkf-fdh)=
gf-gdkfhd-dlgkdf\\
&=gf-g(fdh-dkf)-(gdk-dlg)f=
gf-gfdh+dlgf=\widetilde{gf}\,.\end{split}\]
\end{proof}

\bigskip
\section{The projective model structure on acyclic retractions}

In this section we provide the proof of  Theorem~\ref{thm.modelstructureAR}.
We first observe that the properties MC1 and MC2  of Definition~\ref{def.ModelStructure} follow immediately from the model structure on $\mathbf{coCh}(R)$. 

As regards MC3, denote every acyclic retraction $\xymatrix{M\ar@<.4ex>[r]^\imath&N\ar@<.4ex>[l]^\pi}$ by 
the quadruple $(M,N,\imath,\pi)$, and 
consider a commutative diagram of solid arrows in $\AR(R)$:
 \[ \xymatrix{
 (M_1,N_1,\imath_1,\pi_1) \ar[r]^f \ar[d]_i & (M_3,N_3,\imath_3,\pi_3)\ar[d]^p \\ 
 (M_2,N_2,\imath_2,\pi_2) \ar[r]_g \ar@{.>}[ru]^h & (M_4,N_4,\imath_4,\pi_4)
 }\]
 where $i$ is a cofibration (resp.: trivial cofibration) and $p$ is a trivial fibration (resp.: fibration). The model structure on $\mathbf{coCh}(R)$ ensures the existence of a morphism of cochain complexes $h \colon N_2 \to N_3$ such that $hi=f$ and $ph=g$. 
Denoting  
\[\hat{h}=h-\imath_3\pi_3h-h\imath_2\pi_2+2\imath_3\pi_3h\imath_2\pi_2 \colon 
(M_2,N_2,\imath_2,\pi_2) \to (M_3,N_3,\imath_3,\pi_3),\] 
by the first basic trick~\ref{lem.trick2}, 
the map $\hat{h}$ is a morphism in $\AR(R)$. 
Moreover, since $i,p,f,g$ are morphisms or acyclic retractions, again by Lemma~\ref{lem.trick2} 
we have $\hat{h}i=\widehat{hi}=\hat{f}=f$ and $p\hat{h}=\widehat{ph}=\hat{g}=g$.
and then $\hat{h}$ is the required lifting in the category $\AR(R)$.

Finally, properties MC4 follows from the following two propositions.

\begin{proposition}\label{prop.factorizationAR}
There exists a functorial factorization  
\[(C,FW)\colon \Map(\AR(R)) \to \Map(\AR(R))\times \Map(\AR(R))\]
such that $C(f)$ is a cofibration and $FW(f)$ is a trivial fibration for every morphism $f$.
\end{proposition}

\begin{proof} Consider a morphism $f$ in $\AR(R)$ represented by the commutative diagram 
\[
\xymatrix{
A \ar@/^1.2pc/[rr]^{\Id_A} \ar[r]^{\imath} \ar[d]_{\hat{f}}&B\ar[r]^{\pi} \ar[d]^f & A \ar[d]^{\hat{f}}\\
M \ar@/_1.2pc/[rr]_{\Id_M}\ar[r]_i & N \ar[r]_p & M
}\]
and let $A \xra{g} P \xra{h} M$ be the functorial factorization of 
$\hat{f}\colon A \to M$ in the model category $\mathbf{coCh}(R)$, with  $g$ a cofibration and 
$h$ a trivial fibration. Consider now the pushout $P \amalg_A B$ of $P\xla{g} A \xra{\imath}B$ and the pullback $N \times_MP$ of $N \xra{\;p\;} M \xla{\;h\;} P$. Then we have a commutative  diagram
\[\xymatrix{
A \ar@/^1pc/[rrr]^{\Id_A} \ar[r]_{\imath} \ar@/_1pc/[ddd]_{\hat{f}}\ar[d]^{g}  & B \ar[dddr]|{\hole }^(.35){f} \ar[rr]_{\pi}\ar[d]_{\overline{g}} &  & A \ar@/^1pc/[ddd]^{\hat{f}} \ar[dd]_{g} \\
P\ar[dd]^{h} \ar[rrrd]|{\ }_(.35){\Id_P} \ar[r]^{\overline{\imath}} & P \amalg_A B & && &\\
& &  N\times_M P \ar[r]_{\overline{p}} \ar[d]^{\overline{h}} & P \ar[d]_{h}& \\
M \ar@/_1pc/[rrr]_{\Id_M} \ar[rr]^i  & & N \ar[r]^p & M&}\]
where $\bar{g}$ is a cofibration and $\bar{h}$ is a trivial fibration.
By the universal property of coproducts, there exists a unique morphism 
$\psi_1\colon P \amalg_A B \to P$ such that $g \pi =\psi_1 \overline{g}$ and $\psi_1 \overline{\imath} = \Id_P$. Similarly, there exists a unique morphism   $\psi_2\colon P\amalg_A B \to N$ such that $\psi_2 \overline{g} = f$ and $\psi_2 \overline{\imath} =i h$.
By the universal property of products, there exists a unique morphism $\phi\colon P \amalg_A B \to N\times_MP$ such that $\overline{p} \phi = \psi_1$ and $\overline{h} \phi = \psi_2$. The above diagram becomes:
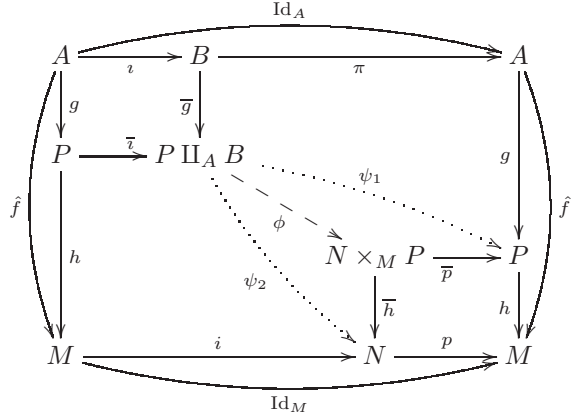
\[\xymatrix{
A \ar@/^1pc/[rrr]^{\Id_A} \ar[r]_{\imath} \ar@/_1pc/[ddd]_{\hat{f}}\ar[d]^{g}  & B \ar[rr]_{\pi}\ar[d]_{\overline{g}}  & & A \ar@/^1pc/[ddd]^{\hat{f}} \ar[dd]_{g} \\
P\ar[dd]^{h} \ar[r]^{\overline{\imath}} & P \amalg_A B \ar@{-->}[rd]_{\phi} \ar@/^/@{.>}[rrd]^{\psi_1} \ar@/_/@{.>}[ddr]_{\psi_2}&  &\\
&  &N\times_M P \ar[r]_{\overline{p}} \ar[d]^{\overline{h}} & P \ar[d]_{h}& \\
M \ar@/_1pc/[rrr]_{\Id_M} \ar[rr]^i && N \ar[r]^p & M&}\]
Let $P\amalg_A B \xra{\gamma} Q \xra{\delta} N\times_M P$ be the functorial  factorization of 
$\phi$, with $\gamma$ a cofibration and and $\delta$ a trivial fibration. Thus we have a commutative diagram
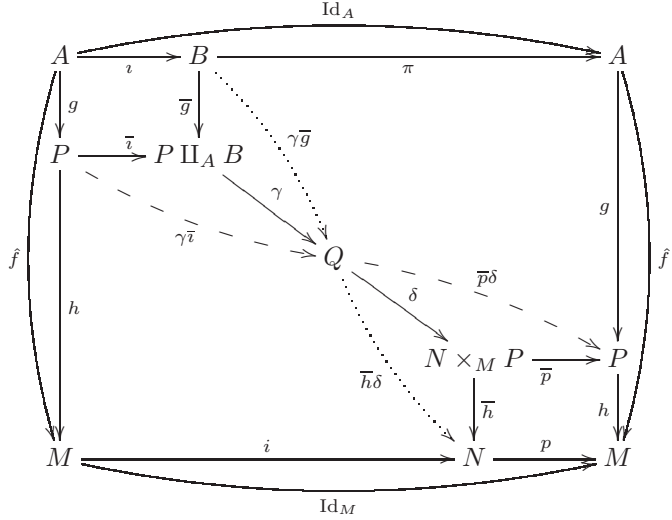
\[\xymatrix{
A \ar@/^1pc/[rrrr]^{\Id_A} \ar[r]_{\imath} \ar@/_1pc/[dddd]_{\hat{f}}\ar[d]^{g}  & B\ar@/^/@{.>}[rdd]^{\gamma \overline{g}} \ar[rrr]_{\pi}\ar[d]_{\overline{g}} & & & A \ar@/^1pc/[dddd]^{\hat{f}} \ar[ddd]_{g} \\
P\ar@/_/@{-->}[rrd]_{\gamma \overline{\imath}}\ar[ddd]^{h} \ar[r]^{\overline{\imath}} & P \amalg_A B \ar[rd]^{\gamma}& & &\\
& & Q \ar[rd]^{\delta} \ar@/_/@{.>}[rdd]_{\overline{h} \delta} \ar@/^/@{-->}[rrd]^{\overline{p}\delta}&& & \\
& & &N\times_M P \ar[r]_{\overline{p}} \ar[d]^{\overline{h}} & P \ar[d]_{h}& \\
M \ar@/_1pc/[rrrr]_{\Id_M} \ar[rrr]^i & && N \ar[r]^p & M&}\]
which reduces to 
\begin{equation}\label{equ.diagrammatondo}
\begin{matrix}
\xymatrix{ A \ar@/_1.2pc/[dd]_{\hat{f}} \ar[d]^{g} \ar[r]_{\imath} \ar@/^1.2pc/[rr]^{\Id_A} & B \ar[d]|{\gamma \overline{g}} \ar[r]_{\pi} & A \ar[d]_{g} \ar@/^1.2pc/[dd]^{\hat{f}} \\
P\ar[r]|{\ \gamma \overline{\imath}\ } \ar[d]^{h} & Q \ar[r]|{\ \overline{p} \delta\ } \ar[d]|{\overline{h} \delta} & P \ar[d]_{h}\\
M\ar@/_1.2pc/[rr]_{\Id_M}\ar[r]^{i} & N \ar[r]^p & M}\end{matrix}\quad.
\end{equation}

Since the construction of Diagram~\eqref{equ.diagrammatondo} is clearly functorial in $\Map(\AR(R))$, in order to conclude the proof it is sufficient to prove that the middle row 
is an acyclic retraction and the middle column is a factorization of $f$ with $\gamma\bar{g}$ a cofibration and $\bar{h}\delta$ a trivial fibration. All of these properties are true because: 
\begin{itemize}
 
 \item $\overline{p}\delta \gamma \overline{\imath}= \Id_P$ by construction.
 
 \item $\delta$ is a trivial fibration by construction and $\bar{p},\bar{h}$ are the pull-backs of the trivial fibrations $p,h$. Hence  $\overline{p} \delta, \bar{h}\delta$,  are trivial fibrations and   
$\gamma \overline{\imath}$ is a weak equivalence by the 2 of 3 property. 
 
\item $\gamma$ is a cofibration by construction and $\bar{g}$ is the push-out of the cofibration $g$.

\end{itemize}
\end{proof}

\begin{proposition}
There exists a functorial factorization  
\[(CW,F)\colon \Map(\AR(R)) \to \Map(\AR(R))\times \Map(\AR(R))\]
such that $CW(f)$ is a trivial cofibration and $F(f)$ is a fibration for every morphism $f$.
\end{proposition}

\begin{proof} Same proof, mutatis mutandis, of Proposition~\ref{prop.factorizationAR}.\end{proof}


\bigskip
\section{The projective model structure on contractions}

In this section we provide the proof of  Theorem~\ref{thm.modelstructureContr}: as in the previous section we notice that  the properties MC1 and MC2  of Definition~\ref{def.ModelStructure} follow immediately from the model structure on $\mathbf{coCh}(R)$. 

In order to prove the lifting property MC3 we shall denote every contraction as a pair $(X,h)$, where 
$X$ is an acyclic retraction and $h$ is a homotopy related to $X$ as in Definition~\ref{def.contraction}.
Consider the following commutative diagram of solid arrow in $\Contr(R)$:
 \[ \xymatrix{
 (A, \alpha ) \ar[r]^f \ar[d]_i & (C, \gamma )\ar[d]^p \\ (B, \beta) \ar[r]_g \ar@{.>}[ru]^h & (E, \delta )
 }\]
where $i$ is a cofibration (resp.: trivial cofibration) and $p$ is a trivial fibration (resp.: fibration). According to Theorem~\ref{thm.modelstructureAR}  there exists a morphism $h \colon B \to C$ of acyclic retractions such that $hi=f$ and $ph=g$. By Lemma~\ref{lem.trick3} the morphism 
$\tilde{h}=h-d\gamma h \beta d   \colon (B,\beta) \to (C,\gamma)$ is a morphism of contractions and  
$\tilde{h}i=\widetilde{hi}=\widetilde{f}=f$ and 
$p\tilde{h}=\widetilde{ph}=\tilde{g}=g$. This proves property MC3 and the remaining part of this section is devoted to the proof of the factorization property MC4.

\begin{definition}\label{def.PathObject} The \emph{path object} functor $P\colon\mathbf{coCh}(R)\to \mathbf{coCh}(R)$ is defined in the following way: for every cochain complex $B=\{B^i\}$ we have 
$P(B)^i=B^i\oplus B^i\oplus B^{i-1}$ and the differential is defined by the formula 
\[\delta\colon P(B)^i\to P(B)^{i+1},\qquad \delta(a,b,c)=(da,db,a-b-dc)\,.\]
\end{definition}

Given $A,B\in \coCh(R)$, $f,g\in \Hom^0_R(A,B)$ and 
$h\in \Hom^{-1}_R(A,B)$, the linear map
\[ A\to P(B),\qquad a\mapsto (f(a),g(a),h(a)),\]
is a morphism of complexes if and only if  $f$ and $g$ are morphisms of complexes and
$f-g=dh+hd$. It follows that the datum $(\xymatrix{M\ar@<.4ex>[r]^\imath&N\ar@<.4ex>[l]^\pi}, h)$ 
is a strong deformation retraction if and only if $\xymatrix{M\ar@<.4ex>[r]^\imath&N\ar@<.4ex>[l]^\pi}$ is an acyclic retraction and 
\[ N\to P(N),\qquad x\mapsto (\imath\pi(x),x,h(x))\]
is a morphism of cochain complexes.

\begin{lemma}\label{lem.PathObject} 
For every cochain complex $B$, the natural projection $P(B)\to B\oplus B$ is  surjective  and the inclusion 
\[ B\to P(B),\qquad b\mapsto (b,b,0),\]
is a quasi-isomorphism.
Moreover if $0\to C\to Q\to N\to 0$ is a short exact sequence of cochain complexes, then 
\[ 0\to C[-1]\to P(Q)\to (Q\oplus Q)\times_{N\oplus N}P(N)\to 0\]
is an exact sequence, where $C[-1]$ is the cochain complex $C$ with the degrees shifted by $1$.
\end{lemma}

\begin{proof} Easy and straightforward.\end{proof}

\begin{lemma}\label{lem.factorizationSDR}
Let $f\colon (\xymatrix{A\ar@<.4ex>[r]^i&B\ar@<.4ex>[l]^p}, k)\to (\xymatrix{M\ar@<.4ex>[r]^\imath&N\ar@<.4ex>[l]^\pi}, h)$ be a morphism of contractions and let 
\[  (\xymatrix{A\ar@<.4ex>[r]^i&B\ar@<.4ex>[l]^p})\xrightarrow{\;\alpha\;}
(\xymatrix{P\ar@<.4ex>[r]^j&Q\ar@<.4ex>[l]^q})\xrightarrow{\;\beta\;} 
(\xymatrix{M\ar@<.4ex>[r]^\imath&N\ar@<.4ex>[l]^\pi})\]
a factorization of $f$ in the model category $\AR(R)$ such that $\alpha$ is a cofibration and $\beta$ is a fibration. If either $\alpha$ or $\beta$ is a weak equivalence, then there exists a homotopy $l\in \Hom_R^{-1}(Q,Q)$ such that 
\begin{equation}\label{equ.factorizationdebole}  
(\xymatrix{A\ar@<.4ex>[r]^i&B\ar@<.4ex>[l]^p},k)\xrightarrow{\;\alpha\;}
(\xymatrix{P\ar@<.4ex>[r]^j&Q\ar@<.4ex>[l]^q},l)\xrightarrow{\;\beta\;} 
(\xymatrix{M\ar@<.4ex>[r]^\imath&N\ar@<.4ex>[l]^\pi},h)\end{equation}
is a factorization in $\mathbf{Contr}(R)$.
\end{lemma}

\begin{proof} By the second basic trick it is sufficient to prove that there exists $l$ such that 
\eqref{equ.factorizationdebole} is a factorization of $f$ in the category of strong deformation retractions.
In the model category $\coCh(R)$ we have a commutative diagram of solid arrows  
\[\xymatrix{
B \ar[d]_{\alpha} \ar[rr]_{(ip,\Id_B)} \ar@/^1.5pc/[rrr]^{(ip,\Id_B,k)} & & B\oplus B 
\ar[d]^(.3){\alpha^{\oplus 2}} & P(B) \ar[d]^{P(\alpha)} \ar[l] \\
Q \ar[d]_{\beta} \ar[rr]_{(jq,\Id_Q)}\ar@{..>}@(ur,ul)[rrr] & & Q\oplus Q \ar[d]^{\beta^{\oplus 2}} & P(Q)\ar[l]\ar[d]^{P(\beta)} \\
N \ar[rr]^{(\imath\pi,\Id_N)} \ar@/_1.5pc/[rrr]_{(\imath\pi,\Id_N,h)} & & N\oplus N & P(N)\ar[l]}\]
and we want to prove that this diagram can be filled with the dotted arrow. This is equivalent to fill with a dotted arrow the solid commutative diagram 
\[ \begin{matrix} \xymatrix{
B \ar@/_2pc/[dddd] \ar[rr]_{(ip,\Id_B)} \ar@/^1pc/[rrr]^{(ip,\Id_B,k)}\ar[dd]_{\alpha} && B \oplus B \ar[dd]^(0.4){\alpha^{\oplus 2}}  & P(B) \ar@/^2pc/[dddd]|(.72){\hole}|(.74){\hole}|(.76){\hole}  \ar[l] \ar[d]_{P(\alpha)} \\
& & & P(Q)\ar[dl] \ar[dd]^{\gamma} \\
Q \ar[rr]^{(jq,\Id_Q)} \ar@/^1.5pc/@{.>}[rrru]^(0.3){\psi} \ar@/_1.5pc/[rrrd]_(0.3){(\phi_1,\phi_2)} \ar[dd]_{\beta} & & Q \oplus Q \ar[dd]^(0.6){\beta^{\oplus 2}} & \\
& & & P(N) \times_{N \oplus N} Q \oplus Q \ar[d] \ar[lu]\\
N \ar@/_1pc/[rrr]_{(\imath\pi,\Id_N,h)}\ar[rr]^{(\imath\pi,\Id_N)} & & N \oplus N & P(N) \ar[l]}
\end{matrix}\qquad
\begin{matrix}\psi=(jq,\Id_Q,l)\\
\\
\phi_1=(\imath\pi,\Id_N,h)\beta\\
\\
\phi_2=(jq,\Id_Q)\end{matrix}.
\]


By Lemma~\ref{lem.PathObject} the morphism $\gamma$ is a fibration that is trivial if and only if $\beta$ is a trivial fibration.  Therefore the dotted lifting $\psi$ exists either when $\alpha$ is a cofibration and $\beta$ a trivial fibration, or when $\alpha$ is a trivial cofibration and $\beta$ a fibration. 
\end{proof}

Finally, properties MC4 follows from the following two propositions.

\begin{proposition}\label{prop.factorizationContr}
There exists a functorial factorization  
\[(C,FW)\colon \Map(\mathbf{Contr}(R)) \to \Map(\mathbf{Contr}(R))\times \Map(\mathbf{Contr}(R))\] 
such that $C(f)$ is a cofibration and $FW(f)$ is a trivial fibration for every morphism $f$.
\end{proposition}

\begin{proof} For every morphism of contractions $f\colon (K,k)\to (H,h)$, $K,H\in \AR(R)$, consider the functorial (C,FW)-factorization $K\xrightarrow{\alpha}L\xrightarrow{\beta}H$ in the model category $\AR(R)$ and chose a homotopy $l$ such that $(K,k)\xrightarrow{\alpha}(L,l)\xrightarrow{\beta}(H,h)$
is a (C,FW)-factorization: the existence of $l$ is provided by Lemma~\ref{lem.factorizationSDR}.
This  defines two functions $C,FW$ on the objects of $\Map(\mathbf{Contr}(R))$, namely 
$C(f)=\alpha$, $FW(f)=\beta$. Now every morphism $\phi$ in $\Map(\mathbf{Contr}(R))$ is given by a commutative square of contractions  
\[ \xymatrix{(K_1,k_1)\ar[d]_{f_1}\ar[r]^{\phi_1}&(K_2,k_2)\ar[d]^{f_2}\\
(H_1,h_1)\ar[r]^{\phi_2}&(H_2,k_2)}\]
which extends to a commutative diagram of acyclic retractions
\[ \xymatrix{(K_1,k_1)\ar[d]_{C(f_1)}\ar[r]^{\phi_1}&(K_2,k_2)\ar[d]^{C(f_2)}\\
(L_1,l_1)\ar[d]_{FW(f_1)}\ar[r]^{\psi}&(L_2,l_2)\ar[d]^{FW(f_2)}\\
(H_1,h_1)\ar[r]^{\phi_2}&(H_2,k_2)}\]
and it is sufficient to consider the morphism of contractions $\tilde{\psi}=\psi-dl_2\psi l_1d$, provided by Lemma~\ref{lem.trick3}, in order to have a functorial factorization in the category $\mathbf{Contr}(R)$.  
\end{proof}

\begin{proposition}
There exists a functorial factorization  
\[(CW,F)\colon \Map(\mathbf{Contr}(R)) \to \Map(\mathbf{Contr}(R))\times \Map(\mathbf{Contr}(R))\] 
such that $CW(f)$ is a trivial cofibration and $F(f)$ is a fibration for every morphism $f$.
\end{proposition}

\begin{proof} Same proof, mutatis mutandis, of Proposition~\ref{prop.factorizationContr}.\end{proof}


\appendix
\bigskip
\section{Semifree extensions}
\label{sec.semifree}

The notion of semifree extension \cite[p. 835]{FHT95} extends the classical notion of semifree module and it 
is very useful in the study of general properties of cofibrations in the projective model structure. 
This appendix is written for reference purposes and contains results which are well known to experts and in any case easy to prove.

\begin{definition}\label{def.semifree}
A morphism $f\colon C\to P$ of cochain complexes over a unitary commutative ring $R$ is called a \textbf{semifree extension} if for every $i\in \Z$ there exists an increasing filtration   
\[ P^i_0\subset P^i_1\subset P^i_2\subset\cdots\]
of $P^i$ such that:
\begin{enumerate} 

\item every $P^i_n$ is an $R$-submodule of $P^i$, $\bigcup_{n\ge 0}P^i_n=P^i$ and $f\colon C^i\to P^i_0$ is an isomorphism;

\item there exists a direct sum decomposition $P_{n+1}^i=P_n^i\oplus A_n^i$, where $A_n^i$ is a free 
$R$-module,  and $d(A_n^i)\subset P^{i+1}_{n}$ for every
$n\ge 0$.

\end{enumerate}
\end{definition}

\begin{example} Let $f\colon C\to P$ be an injective  morphism of cochain complexes such that $f\colon C^i\to P^i$ is an isomorphism 
for every $i>0$ and $P^i/f(C^i)$ is free for every $i$. Then $f$ is a semifree extension. In fact we can consider the filtration 
\[ P^i_n=\begin{cases}P^i&\text{ if }i+n>0\\ f(C^i) &\text{ otherwise.}\end{cases}\]
\end{example}

\begin{theorem}\label{thm.semifreelifting} 
Every semifree extension has the left lifting property with respect to every surjective quasi-isomorphism.
\end{theorem}

\begin{proof} As usual, for every cochain complex $C$ we shall denote by $Z(C),B(C)$ and $H(C)$ the graded modules of cocycles, coboundaries and cohomology of $C$.

Let $C \xra{f} P$ be a semifree extension,  $X \xra{g} Y$ a surjective quasi-isomorphism of cochain complexes, and consider a commutative diagram of solid arrows: 
\[\begin{matrix}\xymatrix{C \ar[r]^{\alpha} \ar[d]_f & X \ar[d]^g \\ P \ar[r]_{\beta} \ar@{.>}[ru]^{h} & Y }
\end{matrix}.\]
Let  $\{ P_n\}_{n \in \N}$ be an exhaustive  filtration of  subcomplexes of $P$  as in Definition~\ref{def.semifree}.
It is sufficient to define recursively a sequence of liftings 
\[\begin{matrix}\xymatrix{C \ar[r]^{\alpha} \ar[d]_f & X \ar[d]^g \\ P_n \ar[r]_{\beta} \ar[ru]^{h_n} & Y }
\end{matrix}\]
such that every $h_n$ extends $h_{n-1}$ and define $h$ as the colimit of $h_n$. Obviously 
$h_0=\alpha f^{-1}$; we may assume $n\ge 0$ and $h_{n}$ already defined.

For every integer $i$, there exists a subset $\{a_j\}_{j\in J^i_n} \subset P_{n+1}^i$ such that 
$da_j\in P_n^{i+1}$ and $P_{n+1}^i$ is the direct summand of $P_n^i$ and the free module generated 
by $\{a_j\}$. By linearity, in order to define $h_{n+1}$ which extends $h_n$ 
it is sufficient to define the elements $h_{n+1}(a_j)$ such that $dh_{n+1}(a_j)=h_n(da_j)$ and 
$gh_{n+1}(a_j)=\beta(a_j)$. Notice that:

 \begin{enumerate}
 \item $d(h_n(da_j))=h_n(d^2 a_j)=0$, and therefore $h_n(da_j) \in Z^{i+1}(X)$;
 
 \item $g(h_n(da_j))= \beta(da_j) = d (\beta(a_j))$, and therefore $g(h_n(da_j))$ is trivial in cohomology.
 \end{enumerate}
Since $g$ is a quasi-isomorphism, also $h_n(da_j)$ is trivial in cohomology and there exists 
$x_j \in X^i$ such that $d(x_j)=h_n(da_j)$. 
  
Moreover since $\beta(a_j)- g(x_j) \in Z^i(Y)$ and $f$ is a surjective quasi-isomorphism  there exists 
$y_j \in Z^i(X)$ such that $g(y_j)=\beta(a_j) - g(x_j)$. It is now sufficient to define 
$h_{n+1}(a_j)=x_j+y_j$. 
\end{proof}

\begin{theorem}\label{thm.semifreefactorization} 
Every morphism $\alpha \colon C\to D$ of cochain complexes of $R$-modules admits a factorization 
$C\xrightarrow{f}P\xrightarrow{g} D$, with  $f$ is a semifree extension and $g$  a surjective quasi-isomorphism.
\end{theorem}

\begin{proof} We construct the factorization by taking an  
increasing sequence of cochain complexes $C=P_0 \subset P_1 \subset P_2\subset \cdots $ and a coherent sequence
of morphisms of cochain complexes $g_n\colon P_n\to D$: coherent means that $g_0=\alpha$ and every $g_n$
extends $g_{n-1}$. The complexes $P_n$ and the morphisms $g_n$ should satisfy the following conditions:   

\begin{itemize}
\item for every $i\in \Z$,  $P^i_{n+1}=P^i_n \oplus A^i_n$ where $A^i_n$ is a free $R$-module such that $d(A_n^i) \subset P^{i+1}_{n}$. This condition implies that the inclusion  $f \colon C \to P=\cup_n P_n$ is a semifree extension. 

\item $g_1\colon Z(P_1) \to Z(D)$ is surjective. This condition implies that $g=\colim g_n\colon P\to D$ is surjective in cohomology.

\item  $g_2\colon P_2 \to D$ is surjective. This condition implies that $g$ is surjective.

\item for every $n>2$, $(g_n)^{-1} (B(D)) \cap Z(P_n) \subset B(P_{n+1}) \cap P_n$. This condition implies that the kernel of $g_n\colon H(P_n)\to H(D)$ is contained in the kernel of $H(P_n)\to H(P_{n+1})$ and therefore that 
$g$ is injective in cohomology, since $Z(P)=\cup_n Z(P_n)$. 
\end{itemize}

The sequence $(P_n,g_n)$ can be constructed recursively in the following way:

\begin{itemize}
\item [$n=0:$] Take  $P_0=C$ and $g_0=\alpha$.

\item [$n=1:$] For every $i\in \Z$, let $A^i_0$ be a free $R$-module such that there exists a surjective map $\pi \colon A^i_0 \to Z^i(D)$. Then  define $P^i_1=P^i_0 \oplus A^n_0$, $g(p+a)=\alpha(p) + \pi(a)$ and 
$d(a)=0$ for every $a\in A^i_0$.

\item[$n=2:$] For every $i\in \Z$, 
let $A^i_1$ be a free $R$-module such that there exists a surjective map $\pi \colon A^i_1 \to D^i$.
If $\{a_j\}_{j\in J}$ is a basis of $A^i_1$, since $g_1\colon Z(P_1)\to Z(D)$ is surjective, there exists a subset 
$\{b_j\}\subset P_1^{i+1}$  such that $g_1(b_j)=d\pi(a_j)$. Then define $P_2^i=P_1^i\oplus A^i_1$, $d(a_j)=b_j$ and 
extend $g_1$ to the map $g_2\colon P_2^i\to D^i$ by setting $g_2(a_j)=\pi(a_j)$.

\item[$n>2:$] For every $i\in \Z$, 
let $A^i_{n-1}$ be a free $R$-module such that there exists a surjective morphism
$\delta \colon A^i_{n-1}\to g_{n-1}^{-1}(B^{i+1}(D))\cap Z^{i+1}(P)$. 
Then define $P^i_n=P^i_{n-1}\oplus A^i_{n-1}$, with differential $d(x+a)=d(x)+\delta(a)$, for 
$x\in P^i_{n-1}$ and $a\in A^i_{n-1}$. If $\{a_j\}_{j\in J}$ is a basis of $A^i_{n-1}$, there exists a subset 
$\{c_j\}\subset D^{i}$ such that $g_{n-1}\delta(a_j)=dc_j$. 
Then we can extend $g_{n-1}$ to a morphism $g_n\colon P_n\to D$ by setting $g_n(a_j)=b_j$. 
\end{itemize}
\end{proof}

\begin{corollary} In the model structure of cochain complexes where  weak equivalences and fibrations are respectively quasi-isomorphisms and surjective maps, a morphism $g\colon C\to D$ is a cofibration if and only if there exists a commutative diagram
\[ \xymatrix{&C\ar[dl]_g\ar[d]^f\ar[dr]^g&\\
D\ar[r]\ar@(dr,dl)[rr]_{\Id}&P\ar[r]&D}\]
with $f$ a semifree extension.
\end{corollary}

\begin{proof} Immediate from the above theorems and the retract argument \cite[Lemma 1.1.9]{Hov99}.
\end{proof}

\end{document}